\documentclass[12pt]{article}
\pdfoutput=1

\usepackage{amsmath, amsthm, amssymb}
\usepackage{ifpdf}
\ifpdf
\usepackage[pdftex]{graphicx}
\else
\usepackage[dvips]{graphicx}
\fi
\usepackage{tikz}
 	 \usetikzlibrary{arrows,backgrounds}
\usepackage[all]{xy}

\input xy
\xyoption{all}

\usepackage[pdftex,plainpages=false,hypertexnames=false,pdfpagelabels]{hyperref}
\newcommand{\arXiv}[1]{\href{http://arxiv.org/abs/#1}{\tt arXiv:\nolinkurl{#1}}}

\newcommand{\googlebooks}[1]{(preview at \href{http://books.google.com/books?id=#1}{google books})}

\usepackage{xcolor}
\definecolor{dark-red}{rgb}{0.7,0.25,0.25}
\definecolor{dark-blue}{rgb}{0.15,0.15,0.55}
\definecolor{medium-blue}{rgb}{0,0,.8}
\hypersetup{
   colorlinks, linkcolor={purple},
   citecolor={medium-blue}, urlcolor={medium-blue}
}

\theoremstyle{plain}
\newtheorem{thm}{Theorem}[section]
\newtheorem*{thm*}{Theorem}
\newtheorem*{cor*}{Corollary}
\newtheorem{cor}[thm]{Corollary}
\newtheorem{lem}[thm]{Lemma}

\theoremstyle{definition}
\newtheorem{defn}[thm]{Definition}
\newtheorem{nota}[thm]{Notation}

\newtheorem{rem}[thm]{Remark}

\DeclareMathOperator{\fdim}{fdim}

\DeclareMathOperator{\Tr}{Tr}
\DeclareMathOperator{\tr}{tr}
\DeclareMathOperator*{\CC}{\mathbb{C}}

\newcommand{\comment}[1]{}

\newcommand{\be}{\begin{enumerate}}
\newcommand{\ee}{\end{enumerate}}

\newcommand{\N}{\mathbb{N}}
\newcommand{\Z}{\mathbb{Z}}

\newcommand{\C}{\mathbb{C}}

\newcommand{\F}{\mathbb{F}}

\newcommand{\I}{\infty}

\newcommand{\cA}{\mathcal{A}}

\newcommand{\cM}{\mathcal{M}}
\newcommand{\cN}{\mathcal{N}}
\newcommand{\cP}{\mathcal{P}}

\newcommand{\noshow}[1]{}

\begin{document}
\title{Free product von Neumann algebras associated to graphs, and Guionnet, Jones, Shlyakhtenko subfactors in infinite depth}
\author{ Michael Hartglass }
\date{\today}
\maketitle

\begin{abstract}
Given a  subfactor planar algebra $\cP$, Guionnet, Jones and Shlyakhtenko give a diagrammatic construction of a $II_{1}$ subfactor whose planar algebra is $\cP$.  They showed if $\cP$ is finite-depth, then the factors are interpolated free group factors, and they identified the parameters.  We prove if $\cP$ is infinite-depth, then the factors are isomorphic to $L(\F_{\I})$.
\end{abstract}

%%%%%%%%%%%%%%%%%%%%%%%%%%%%%%%%%%%%%%%%%%%%%%%%%%

\section{Introduction}

In \cite{MR696688}, Jones initiated the study of modern subfactor theory.  Given a finite index $II_{1}$ subfactor $A_{0} \subset A_{1}$, one computes its standard invariant: two towers $(A_{0}' \cap A_{j}: j \geq 0)$  and $(A_{1}' \cap A_{j}: j \geq 1)$ of finite dimensional von Neumann algebras \cite{MR696688}. The standard invariant has been axiomatized by Ocneanu's paragroups \cite{MR996454}, Popa's $\lambda-$lattices \cite{MR1334479}, and Jones' subfactor planar algebras \cite{ArXiv:1007.1158}. Popa showed that given a standard invariant $\cP$, we can reconstruct a $II_{1}$ subfactor $A_{0} \subset A_{1}$ whose standard invariant is $\cP$ \cite{MR1334479}. Guionnet, Jones, and Shlyakhtenko \cite{MR2732052} give a planar-algebraic proof of the above result.  Moreover, if $\cP$ is finite depth with loop paramater $\delta > 1$, they showed that $A_{k}$, the $k^{th}$ factor in the Jones tower, is isomorphic to $L(\F(1 + 2\delta^{-2k}(\delta - 1)I))$ where $I$ is the global index of $\cP$ \cite{MR2807103}.   Kodiyalam and Sunder also obtained this formula when $\cP$ is depth 2 \cite{MR2574313, MR2557729}.  In this paper, we prove the following theorem:

\begin{thm*}
If $\cP$ is infinite depth, then every factor in the construction of \cite{MR2732052} is isomorphic to $L(\F_{\I})$.
\end{thm*}

 Using this theorem, we recover a diagrammatic proof of a result of Popa and Shlyakhtenko for $\cP$ infinite depth \cite{MR2051399}:

 \begin{cor*}
 Every infinite depth subfactor planar algebra is the standard invariant of $\cN \subset \cM$ where $\cN, \cM \cong L(\F_{\I})$.
 \end{cor*}

\subsection{Outline of the proof}

To prove the above theorem, we will bootstrap the proof from \cite{MR2807103} of the finite-depth case to the infinite-depth case using standard embedding tricks.

\paragraph {The GJS Construction:}

      We will recall the construction of Guionnet, Jones and Shlyakhtenko. For more details, see \cite{MR2732052} and \cite{MR2807103}. Let $\cP = (\cP_{n}^{\pm})_{n \geq 0}$ be a subfactor planar algebra with loop parameter $\delta > 1$ (see \cite{ArXiv:1007.1158} for the definition of a subfactor planar algebra).  Set $Gr_{k}(\cP^{+})= \bigoplus_{n \geq 0} \cP_{k,n}^{+}$ where $\cP^{+}_{k,n} = \cP^{+}_{k+n}$ and an element of $\cP^{+}_{k,n}$ is represented as
       $$
    
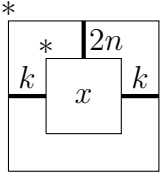
\begin{tikzpicture}
       \draw (-.5, -.5) -- (-.5, .5) -- (.5, .5) -- (.5, -.5) -- (-.5, -.5);
       \draw (-1, -1) -- (-1, 1) -- (1, 1) -- (1, -1) -- (-1, -1);
       \draw[ultra thick] (0, 1) -- (0, .5);
       \draw[ultra thick] (1, 0) -- (.5, 0);
       \draw[ultra thick] (-1, 0) -- (-.5, 0);
       \node at (0, 0) {$x$};
       \node at (-.75, .25) {$k$};
       \node at (.75, .25) {$k$};
       \node at (.30, .75) {$2n$};
       \node at (-.5, .6) {*};
       \node at (-1, 1.1) {*};
       \end{tikzpicture}
       $$
       where the * is always in an unshaded region and a thick string with a $j$ next to it denotes $j$ parallel strings.  If $x \in \cP_{n+k}^{+}$ and $y \in \cP_{m + k}^{+}$ then define a multiplication $\wedge_{k}$ by
       $$
        x \wedge_{k} y =
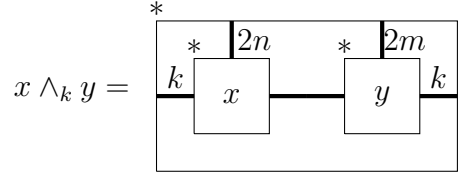
       \begin{tikzpicture} [baseline = 0 cm]
       \draw (-1.5, -.5) -- (-1.5, .5) -- (-.5, .5) -- (-.5, -.5) -- (-1.5, -.5);
       \draw[ultra thick] (-1, 1) -- (-1, .5);
       \draw[ultra thick] (0, 0) -- (-.5, 0);
       \draw[ultra thick] (-2, 0) -- (-1.5, 0);
       \node at (-1, 0) {$x$};
       \node at (-1.75, .25) {$k$};
       \node at (-.7, .75) {$2n$};
       \node at (-1.5, .6) {*};
       \draw (.5, -.5) -- (.5, .5) -- (1.5, .5) -- (1.5, -.5) -- (.5, -.5);
       \draw[ultra thick] (1, 1) -- (1, .5);
       \draw[ultra thick] (2, 0) -- (1.5, 0);
       \draw[ultra thick] (0, 0) -- (.5, 0);
       \node at (1, 0) {$y$};
       \node at (1.75, .25) {$k$};
       \node at (1.30, .75) {$2m$};
       \node at (.5, .6) {*};
       \draw (-2, -1) -- (2,-1)--(2,1)--(-2, 1)--(-2, -1);
       \node at (-2, 1.1) {*};
       \end{tikzpicture}
       $$
        which is an element in $\cP_{k, m + n}^{+}$.  One can endow $Gr_{k}(\cP^{+})$ with the following trace:  if $x \in \cP^{+}_{k, n}$ then
        $$
        \tr(x) = \delta^{-k} \cdot 
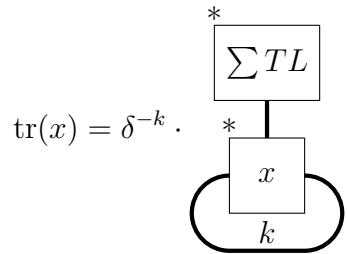
\begin{tikzpicture}[baseline=.5cm]
	\draw (.5,.5) --(-.5,.5) -- (-.5,-.5) -- (.5,-.5) -- (.5, .5) ;
    \draw (.7,2)--(-.7,2)--(-.7,1)--(.7, 1)--(.7,2);
\draw[ultra thick] (-.5, 0) arc (90:270: .5cm) -- (.5, -1) arc(-90:90: .5cm);
\draw[ultra thick] (0, .5) -- (0, 1);
\node at (0, 0) {$x$};
\node at (-.5, .6) {*};
\node at (-.7, 2.1) {*};
\node at (0, 1.5) {$\sum TL$};
\node at (0, -.75) {$k$}; \end{tikzpicture}
$$
 where $\sum TL$ denotes the sum of all Temperely-Lieb diagrams, i.e. all planar pairings of the $2n$ strings on top of $x$.  This trace is positive definite, and one can form the von Neumann algebra $A_{k}$ which is the strong closure of $Gr_{k}(\cP^{+})$ acting on $L^{2}(Gr_{k}(\cP^{+}))$ by left multiplication (under $\wedge_{k}$).  It is shown that $A_{k}$ is a $II_{1}$ factor.  Moreover one can view $x \in A_{k}$ as an element in $A_{k+1}$ as follows:
 $$
    
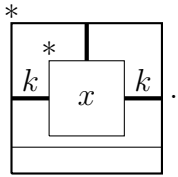
\begin{tikzpicture} [baseline = 0 cm]
       \draw (-.5, -.5) -- (-.5, .5) -- (.5, .5) -- (.5, -.5) -- (-.5, -.5);
       \draw[ultra thick] (0, 1) -- (0, .5);
       \draw[ultra thick] (1, 0) -- (.5, 0);
       \draw[ultra thick] (-1, 0) -- (-.5, 0);
       \node at (0, 0) {$x$};
       \node at (-.75, .25) {$k$};
       \node at (.75, .25) {$k$};
       \node at (-.5, .6) {*};
       \node at (-1, 1.1) {*};
       \draw[thick] (-1, -1) -- (-1, 1) -- (1, 1) -- (1, -1) -- (-1, -1);
       \draw (-1, -.65) -- (1, -.65);
    \end{tikzpicture} \, .
    $$
   With this identification, $A_{k}$ is a von Neumann subalgebra of $A_{k+1}$ and $A_{0} \subset A_{1} \subset \cdots \subset A_{k} \subset \cdots$ is a Jones tower of $II_{1}$ factors with standard invariant $\cP^{+}$.

   To identify the isomorphism type of the $A_{k}$, we look at the semi-finite algebra  $$V_{+} = \bigoplus_{k + l + m \textrm{ even }} \cP^{+}_{k,l,m}$$ where $\cP^{+}_{k,l,m} = \cP^{+}_{\frac{k + l + m}{2}}$ and is spanned by boxes of the form
   $$
    
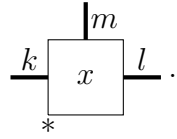
\begin{tikzpicture} [baseline = 0 cm]
       \draw (-.5, -.5) -- (-.5, .5) -- (.5, .5) -- (.5, -.5) -- (-.5, -.5);
       \draw[ultra thick] (0, 1) -- (0, .5);
       \draw[ultra thick] (1, 0) -- (.5, 0);
       \draw[ultra thick] (-1, 0) -- (-.5, 0);
       \node at (0, 0) {$x$};
       \node at (-.75, .25) {$k$};
       \node at (.75, .25) {$l$};
       \node at (.25, .75) {$m$};
       \node at (-.5, -.7) {*};
       \end{tikzpicture} \, .
       $$
       The element $x$ above is identified with the following element of $\cP_{k+2p, l+2q, m}^{+}$:
       \begin{equation}
       \label{r}
  \delta^{-(p+q)/2}  \begin{tikzpicture}[baseline = -.5cm]
       \draw (-.5, -.5) -- (-.5, .5) -- (.5, .5) -- (.5, -.5) -- (-.5, -.5);
       \draw[ultra thick] (0, 1) -- (0, .5);
       \draw[ultra thick] (1, 0) -- (.5, 0);
       \draw[ultra thick] (-1, 0) -- (-.5, 0);
       \draw[thick] (-1,1)--(-1,-2)--(1, -2)--(1,1)--(-1,1);
       \draw (-1, -.5) arc(90:-90 : .2cm);
       \draw (-1, -1.5) arc(90:-90 : .2cm);
       \draw (1, -0.5) arc(90:270 : .2cm);
       \draw (1, -1.5) arc(90:270 : .2cm);
       \node at (0, 0) {$x$};
       \node at (-.75, .25) {$k$};
       \node at (.75, .25) {$l$};
       \node at (.25, .75) {$m$};
       \node at (-1, -2.2) {*};
       \node at (-.85, -1.1) {\vdots};
       \node at (.85, -1.1) {\vdots};
       \node at (-.7, -1.1) {$p$};
       \node at (.7, -1.1) {$q$};
       \node at (-.5, -.7) {*};
       \end{tikzpicture}
       \end{equation}
        where there are $p$ cups on the left and $q$ cups on the right.  Under these identifications, $V_{+}$ completes to a semifinite von Neumann algebra, $\cM_{+}$ where the multiplication is given by
        $$
        \left(\begin{tikzpicture} [baseline = 0cm]
       \draw (-.5, -.5) -- (-.5, .5) -- (.5, .5) -- (.5, -.5) -- (-.5, -.5);
       \draw[ultra thick] (0, 1) -- (0, .5);
       \draw[ultra thick] (1, 0) -- (.5, 0);
       \draw[ultra thick] (-1, 0) -- (-.5, 0);
       \node at (0, 0) {$x$};
       \node at (-.75, .25) {$k$};
       \node at (.75, .25) {$l$};
       \node at (.25, .75) {$m$};
       \node at (-.5, -.7) {*};
       \end{tikzpicture}\right) \cdot \left(\begin{tikzpicture}[baseline = 0cm]
       \draw (-.5, -.5) -- (-.5, .5) -- (.5, .5) -- (.5, -.5) -- (-.5, -.5);
       \draw[ultra thick] (0, 1) -- (0, .5);
       \draw[ultra thick] (1, 0) -- (.5, 0);
       \draw[ultra thick] (-1, 0) -- (-.5, 0);
       \node at (0, 0) {$y$};
       \node at (-.75, .25) {$k'$};
       \node at (.75, .25) {$l'$};
       \node at (.25, .75) {$m'$};
       \node at (-.5, -.7) {*};
       \end{tikzpicture}\right) = \delta_{l, k'} \begin{tikzpicture}[baseline = 0cm]
       \draw (-1.5, -.5) -- (-1.5, .5) -- (-.5, .5) -- (-.5, -.5) -- (-1.5, -.5);
       \draw[ultra thick] (-1, 1) -- (-1, .5);
       \draw[ultra thick] (0, 0) -- (-.5, 0);
       \draw[ultra thick] (-2, 0) -- (-1.5, 0);
       \node at (-1, 0) {$x$};
       \node at (-1.75, .25) {$k$};
       \node at (-.7, .75) {$n$};
       \node at (-1.5, -.7) {*};
       \draw (.5, -.5) -- (.5, .5) -- (1.5, .5) -- (1.5, -.5) -- (.5, -.5);
       \draw[ultra thick] (1, 1) -- (1, .5);
       \draw[ultra thick] (2, 0) -- (1.5, 0);
       \draw[ultra thick] (0, 0) -- (.5, 0);
       \node at (1, 0) {$y$};
       \node at (1.75, .25) {$l'$};
       \node at (1.30, .75) {$m'$};
       \node at (0, .3) {$l$};
       \node at (.5, -.7) {*};
       \draw (-2, -1) -- (2,-1)--(2,1)--(-2, 1)--(-2, -1);
       \node at (-2, -1.2) {*};
       \end{tikzpicture}
       $$
         where we have assumed that we have added enough cups as in diagram \eqref{r} so that $l$ and $k'$ are either the same or differ by 1.  The trace on $\cM_{+}$ is given by
         $$
         \Tr(x) = \begin{tikzpicture}[baseline=.5cm]
	\draw (.5,.5) --(-.5,.5) -- (-.5,-.5) -- (.5,-.5) -- (.5, .5) ;
    \draw (.7,2)--(-.7,2)--(-.7,1)--(.7, 1)--(.7,2);
\draw[ultra thick] (-.5, 0) arc (90:270: .5cm) -- (.5, -1) arc(-90:90: .5cm);
\draw[ultra thick] (0, .5) -- (0, 1);
\node at (0, 0) {$x$};
\node at (-.5, -.7) {*};
\node at (-.7, 2.1) {*};
\node at (0, 1.5) {$\sum TL$};
\node at (0, -.75) {$k$}; \end{tikzpicture}
$$
 provided that the number of strings on the left and right of $x$ have the same parity and is zero otherwise.  It is easy to check that the identification in diagram \eqref{r} respects both the trace and multiplication.

The algebras $A_{2k}$ above are a compression of $\cM_{+}$ by the projection $p_{2k}^{+}$ where for general $n$,
$$
p_{n}^{+} = \begin{tikzpicture}[baseline = 0cm]
       \draw[thick] (-.5, -.5) -- (-.5, .5) -- (.5, .5) -- (.5, -.5) -- (-.5, -.5);
       \draw[ultra thick] (-.5, 0) -- (.5, 0);
              \node at (0, .2) {$n$};
       \node at (-.5, -.7) {*};
       \end{tikzpicture}
       $$
        Similarly, we can consider a semifinite von Neumann algebra $\cM_{-}$ generated by the $\cP_{n}^{-}$'s (where the * is in a \emph{shaded} region and on the bottom of the box), and if we define projections $p_{n}^{-}$, then $A_{2k+1}$ is the compression of $\cM_{-}$ by $p_{2k+1}^{-}$.

        A diagrammatic argument shows that $\cM_{+}$ is generated by
        $$
        \cA_{+} = \displaystyle \left(\bigcup_{k, \ell} \cP^{+}_{k, \ell, 0}\right)^{''} \textrm{ and } X = \textrm{s}-\lim_{k \rightarrow \infty} \begin{tikzpicture}[baseline = 0cm]
       \draw[thick] (-.7, -.7) -- (-.7, .7) -- (.7, .7) -- (.7, -.7) -- (-.7, -.7);
       \draw (-.7, .4)--(0, .4) arc(-90:0 : .3cm);
       \draw[ultra thick] (-.7, 0) -- (.7, 0);
       \node at (0, -.3) {$2k$};
       \node at (-.7, -.9) {*};
       \node at (-.7, -.9) {*};
       \end{tikzpicture}
         +
        \begin{tikzpicture} [baseline = 0cm]
       \draw[thick] (-.7, -.7) -- (-.7, .7) -- (.7, .7) -- (.7, -.7) -- (-.7, -.7);
       \draw (.7, .4)--(0, .4) arc(-90:-180 : .3cm);
       \draw[ultra thick] (-.7, 0) -- (.7, 0);
       \node at (0, -.3) {$2k$};
       \node at (-.7, -.9) {*};
       \end{tikzpicture}
       $$
where the limit above is in the strong operator topology. This element is an $\cA_{+}$-valued semicircular element in the sense of \cite{MR1704661} and is used in the calculation of the isomorphism class of the algebras $A_{k}$

   \paragraph{The finite depth case:}  Let $\Gamma$ denote the principal graph of $\cP$ with edge set $E(\Gamma)$ and initial vertex *.  Let $\ell^{\infty}(\Gamma)$ as the von Neumann algebra of bounded functions on the vertices of $\Gamma$ and endow $\ell^{\infty}(\Gamma)$ with a trace $\tr$ such that $\tr(p_{v}) = \mu_{v}$, where $p_{v}$ is the delta function at $v$ and $\mu_{v}$ is the entry corresponding to a fixed Perron-Frobenius eigenvector for $\Gamma$ with $\mu_* = 1$.  From \cite{MR2807103}, $A_{0} = p_{*}\cM(\Gamma)p_{*}$ where $\cM(\Gamma)$ is the von Neumann algebra generated by $(\ell^{\I}(\Gamma), \tr)$ and $\ell^{\infty}(\Gamma)$-valued semicircular elements $\{X_{e}: e \in E(\Gamma)\}$ which are compressions of $X$ by partial isometries in $\cA_{+}$ and are free with amalgamation over $\ell^{\infty}(\Gamma)$.  Each $X_{e}$ is supported under $p_{v} + p_{w}$, where $e$ connects $v$ and $w$, and we have $X_{e} = p_{v}X_{e}p_{w} + p_{w}X_{e}p_{v}$.  Assuming that $\mu_v \geq \mu_w$, the scalar-valued distribution of $X_{e}^{2}$ in $(p_v + p_w)\cM(\Gamma)(p_v + p_w)$ is free-Poisson with an atom of size $\frac{\mu_v - \mu_w}{\mu_v + \mu_w}$ at 0.  Therefore,
  $$
  vN(\ell^{\infty}(\Gamma), X_{e}) = L(\Z) \otimes M_{2}(\C) \oplus \C \oplus \ell^{\I}(\Gamma \setminus \{v, w\})
  $$
   with $p_w = (1\otimes e_{1,1})\oplus 0\oplus 0$ and $p_v = (1\otimes e_{2,2})\oplus 1\oplus 0$, where $\{e_{i,j}: 1 \leq i,j \leq 2\}$ is a system of matrix units for $M_{2}(\C)$.  If $\Gamma$ is finite, Dykema's formulas for computing certain amalgamated free products \cite{MR1201693, MR2765550} show that $\cM(\Gamma)$ is an interpolated free group factor and the compression formula gives the result for $A_{0}$.  Since $A_{2n}$ is a $\delta^{2n}-$amplification of $A_{0}$, the result holds for $A_{2n}$.  The factor $A_{1}$ is a compression of $M(\Gamma^{*})$ with $\Gamma^{*}$ the dual principal graph of $\cP$.  Applying the same analysis to $\Gamma^{*}$ gives the formula for the $A_{2n+1}$'s.

   \paragraph{The infinite depth case:}

    We similarly define $\cM(\Gamma)$ for an arbitrary connected, loopless (not necessarily bipartite) graph $\Gamma$.  If $\Gamma$ is finite, we show that $\cM(\Gamma) \cong L(\F_{t}) \oplus A$ where $A$ is finite-dimensional and abelian ($A$ can possibly be $\{0\}$).  Furthermore, if $p_{\Gamma}$ is the identity of $L(\F_{t})$ and $\Gamma'$ is a finite graph containing $\Gamma$, then the inclusion $p_{\Gamma}\cM(\Gamma)p_{\Gamma} \rightarrow p_{\Gamma}\cM(\Gamma')p_{\Gamma}$ is a standard embedding of interpolated free group factors (see Definition \ref{defn:Dyk} and Remark \ref{rem:Dyk} below).  Therefore, if $\cP$ is infinite depth with principal graph $\Gamma$, we write $\Gamma$ as an increasing union of finite graphs $\Gamma_{k}$ where $\Gamma_{k}$ is $\Gamma$ truncated at depth $k$.  Since standard embeddings are preserved by cut-downs, the inclusion $p_{*}\cM(\Gamma_{k})p_{*} \rightarrow p_{*}\cM(\Gamma_{k+1})p_{*}$ is a standard embedding.  As $A_{0}$ is the inductive limit of the $p_{*}\cM(\Gamma_{k})p_{*}$'s, it is an interpolated free group factor where the parameter is the limit of the parameters for the $p_{*}\cM(\Gamma_{k})p_{*}$'s, which is $\I$.  Since the factors $A_{2k}$ are amplifications of $A_{0}$, $A_{2k} \cong L(\F_{\infty})$.  Applying the same analysis to $\Gamma^{*}$ (the dual principal graph of $\cP$) shows that $A_{2k+1} \cong L(\F_{\infty})$.

 \paragraph{Organization:}  Section \ref{sec:preliminaries} covers some preliminary material on interpolated free group factors, free dimension, and standard embeddings.  Section \ref{sec:VNGraph} introduces $\cM(\Gamma)$ and establishes both its structure and how it includes into $\cM(\Gamma')$ for $\Gamma$ a subgraph of $\Gamma'$.  Section \ref{sec:GJSInfinite} provides the proof that the factors $A_{k}$ above are all isomorphic to $L(\F_{\I})$.

\paragraph{Acknowledgements:}  The author would like to thank Arnaud Brothier, Vaughan Jones, David Penneys, and Dimitri Shlyakhtenko for their conversations and encouragement.  The author was supported by NSF Grant DMS-0856316 and DOD-DARPA grants HR0011-11-1-0001 and HR0011-12-1-0009.

%%%%%%%%%%%%%%%%%%%%%%%%%%%%%%%%%%%%%%%%%%%%%%%%%%
\section{Preliminaries}\label{sec:preliminaries}

Dykema \cite{MR1256179} and R\u{a}dulescu \cite{MR1258909} independently developed interpolated free group factors $L(\F_{t})$ for $1 < t \leq \I$.  These coincide with the usual free group factors when $t \in \N \cup \{\infty\}$ and they satisfy
$$
L(\F_{t})*L(\F_{s}) = L(\F_{s+t}) \textrm{ and } L(\F_{t})_{\gamma} = L(\F(1 + \gamma^{-2}(t-1))),
$$
where $M_{\gamma}$ is the $\gamma-$amplification of the $II_{1}$ factor $M$. It is known that either the interpolated free group factors are all isomorphic or they are pairwise non-isomorphic \cite{MR1256179, MR1258909}.

\begin{nota} \label{nota:vNA} Throughout this paper, we will be concerned with finite von Neumann algebras $(\cM, tr)$ which can be written in the form
$$
 \cM = \overset{p_{0}}{\underset{\gamma_{0}}{\cM_{0}}} \oplus \bigoplus_{j \in J} \overset{p_{j}}{\underset{\gamma_{j}}{L(\F_{t_{j}})}} \oplus  \bigoplus_{k \in K} \overset{q_{k}}{\underset{\alpha_{k}}{M_{n_{k}}}}
 $$
 where $\cM_{0}$ is a diffuse hyperfinite von Neumann algebra, $L(\F_{t_{j}})$ is an interpolated free group factor with parameter $t_{j}$, $M_{n_{k}}$ is the algebra of $n_{k} \times n_{k}$ matrices over the scalars, and the sets $J$ and $K$ are at most finite and countably infinite respectively.  We use $p_{j}$ to denote the projection in $L(\F_{t_{j}})$ corresponding to the identity of $L(\F_{t_{j}})$ and $q_{k}$ to denote a minimal projection in $M_{n_{k}}$. The projections $p_{j}$ and $q_{k}$ have traces $\gamma_{j}$ and $\alpha_{k}$ respectively.  Let $p_{0}$ be the identity in $\cM_{0}$ with trace $\gamma_{0}$.   We write $\overset{p,q}{M_{2}}$ to mean $M_{2}$ with a choice of minimal orthogonal projections $p$ and $q$. \end{nota}

If the interpolated free group factors turn out to be non-isomorphic, it is desirable to be able to calculate which interpolated free group factors appear in amalgamated free products.  To help facilitate this calculation, Dykema defined the notion of free dimension.  In general, one has
$$
\fdim(\cM_{1} \underset{D}* \cM_{2}) = \fdim(\cM_{1}) + \fdim(\cM_{2}) - \fdim(D)
$$
whenever $\cM_{1}$ and $\cM_{2}$ are of the form of Notation \ref{nota:vNA} and $D$ is finite dimensional \cite{MR1201693, MR1363079,MR2765550,arXiv:1110.5597v1}.  In general, for the algebra $\cM$ in Notation \ref{nota:vNA}, we have $$\fdim(\cM) = 1 + \sum_{j \in J}\gamma_{j}^{2}(t_{j} - 1) - \sum_{k \in K}\alpha_{k}^{2}.$$
Notice that this includes the special case $\fdim(L(\F_{t})) = t$.

Of course if the interpolated free group factors are isomorphic, then the free dimension is not well defined; however, the only purpose of the free dimension is to determine the parameter of interpolated free group factors which show up in amalgamated free products.  Therefore all the lemmas below will remain valid if all references to free dimension are removed.

Of critical importance will be the notion of a \emph{standard embedding} of interpolated free group factors \cite{MR1201693}.  This is a generalization of a mapping $\F_{n} \rightarrow \F_{m}$ for $m > n$ sending the $n$ generators of $\F_{n}$ onto $n$ of the $m$ generators for $\F_{m}$.

\begin{defn} \label{defn:Dyk}

Let $1 < r < s$ and $\phi : L(\F_{r}) \rightarrow L(\F_{s})$ be a von Neumann algebra homomorphism.  We say that $\phi$ is a standard embedding if there exist nonempty sets $S \subset S'$, a family of projections $\{p_{s'}: s' \in S'\}$ with $p_{s'} \in R$ (the hyperfinite $II_{1}$ factor), a free family $\{X^{s'}: s' \in S'\}$ of semicircular elements which are also free from $R$, and isomorphisms
 $$
 \alpha: L(\F_{r}) \rightarrow (R \cup \{p_{s}X^{s}p_{s}\}_{s \in S})'' \textrm{ and } \beta : L(\F_{s}) \rightarrow (R \cup \{p_{s'}X^{s'}p_{s'}\}_{s' \in S'})''
 $$
 such that $\phi = \beta^{-1} \circ \iota \circ \alpha$ where $\iota: (R \cup \{p_{s}X^{s}p_{s}\}_{s \in S})'' \rightarrow (R \cup \{p_{s'}X^{s'}p_{s'}\}_{s' \in S'})''$ is the canonical inclusion.  We will write $A \overset{s.e.}{\hookrightarrow} B$ to mean that the inclusion of $A$ into $B$ is a standard embedding.
\end{defn}

\begin{rem} \label{rem:Dyk} Dykema in \cite{MR1201693} and \cite{MR1363079} shows the following useful properties of standard embeddings which we will use extensively in this paper.

\begin{itemize}

\item[(1)] If $A$ is an interpolated free group factor, the canonical inclusion $A \rightarrow A * \cM$ is a standard embedding whenever $\cM$ is of the form in Notation \ref{nota:vNA}.
\item[(2)] A composite of standard embeddings is a standard embedding.
\item[(3)] If $A_{n} = L(\F_{s_{n}})$ with $s_{n} < s_{n+1}$ for all $n$ and $\phi_{n}: A_{n} \overset{s.e.}{\hookrightarrow} A_{n+1}$, then the inductive limit of the $A_{n}$ with respect to the $\phi_{n}$ is $L(\F_{s})$ where $s = \displaystyle \lim_{n \rightarrow \infty}s_{n}$.
\item[(4)] If $t > s$  then $\phi: L(\F_{s}) \overset{s.e.}{\hookrightarrow} L(\F_{t})$ if and only if for any nonzero projection $p \in L(\F_{s})$, $\phi|_{pL(\F_{s})p}: pL(\F_{s})p \overset{s.e.}{\hookrightarrow} \phi(p)L(\F_{t})\phi(p)$.

\end{itemize}

\end{rem}

Our work will rely heavily on the following two lemmas.
\begin{lem}  [\cite{arXiv:1110.5597v1}] \label{lem:DR1}

Let $\cN = (\overset{p}{M_{n}(\C)} \oplus B) \underset{D}{*} C$ and $\cM = (M_{n}(\C) \otimes A \oplus B) \underset{D}{*} C$ where $A$, $B$ and $C$ are finite von Neumann algebras and $D$ is a finite dimensional abelian von Neumann algebra.  Let $E$ be the trace-preserving conditional expectation of $\cM$ onto $D$.  Assume $p$ lies under a minimal projection in $D$ and $E|_{M_{n}(\C) \otimes A} = E|_{M_{n}(\C)} \otimes tr_{A}$.  Then $p\cM p = p\cN p * A$ and the central support of $p$ in $\cM$ is the same as that in $\cN$.

\end{lem}

\begin{lem}  [\cite{arXiv:1110.5597v1}] \label{lem:DR2}

Let $\cN = (\overset{p}{\underset{\gamma}{M_{m}(\C)}} \oplus \overset{q}{\underset{\gamma}{M_{n-m}(\C)}} \oplus B) \underset{D}{*} C$ and $\cM = (\underset{\gamma}{M_{n}(\C)} \oplus B) \underset{D}{*} C$ with $B$, $C$, $D$ as in Lemma \ref{lem:DR1}.  Assume $p$ and $q$ sit under minimal projections in $D$ and $p$ is equivalent to $q$ in $\cN$.  Then $p\cM p = p\cN p * L(\Z)$ and the central support of $p$ in $\cM$ is the same as that in $\cN$.

\end{lem}

 Note that if $A$, $B$ and $C$ are in the form in Notation \ref{nota:vNA}, and if $\cN$ is an interpolated free group factor, then the proofs of the above lemmas in \cite{arXiv:1110.5597v1} show that $p\cN p \rightarrow p\cM p$ of Lemmas \ref{lem:DR1} and \ref{lem:DR2} are standard embeddings.  This implies $\cN \overset{s.e.}{\hookrightarrow} \cM$ by Remark \ref{rem:Dyk}.

%%%%%%%%%%%%%%%%%%%%%%%%%%%%%%%%%%%%%%%%%%%%%%%%%%

%%%%%%%%%%%%%%%%%%%%%%%%%%%%%%%%%%%%%%%%%%%%%%%%%%
\section{A von Neumann algebra associated to a finite connected graph}\label{sec:VNGraph}

Let $\Gamma$ be a connected, loopless finite graph with edge set $E(\Gamma)$ and vertex set $V(\Gamma)$.  Assume further that each vertex $v \in V(\Gamma)$ is weighted by a real constant $\gamma_{v} > 0$ with $\sum_{v \in \Gamma} \gamma_{v} = 1$ (the weighting does \emph{not} have to be the Perron-Frobenius weighting).  Consider the abelian von Neumann algebra $\ell^{\infty}(\Gamma)$. Let $p_{v}$ be the indicator function on the vertex $v$ and its trace will be $\gamma_{v}$.  We construct a finite von Neumann algebra associated to $\Gamma$ (also see \cite{MR2772347}).

\begin{defn}\label{defn:vNGraph}
Let $\Gamma$ be as above, $e$ be an edge in $\Gamma$ connecting the vertices $v$ and $w$, and assume $\gamma_{v} \geq \gamma_{w}$.  Define
$$
\cA_{e} = \underset{2\gamma_{w}}{M_{2}(\C) \otimes L(\Z)} \oplus \CC_{\gamma_{v} - \gamma_{w}}^{p^{e}_{v}} \oplus \ell^{\infty}(\Gamma \setminus \{v, w\})
$$
 where the trace on $M_{2} \otimes L(\Z)$ is $\tr_{M_{2}} \otimes \tr_{L(\Z)}$.  $\cA_{e}$ includes $\ell^{\infty}(\Gamma)$ by setting \begin{align*}p_{w} &= 1 \otimes e_{1, 1} \oplus 0 \oplus 0 \textrm{ and } \\ p_{v} &= 1 \otimes e_{2, 2} \oplus 1 \oplus 0.\end{align*}   Therefore, the trace preserving conditional expectation $E_{e}: \cA_{e} \rightarrow \ell^{\infty}(\Gamma)$ has the property $E_{e}|_{M_{2} \otimes L(Z)} = E_{e}|_{M_{2}} \otimes \tr|_{L(\Z)}$. We define $\cM(\Gamma)$, the von Neumann algebra associated to $\Gamma$, by $$
 \cM(\Gamma) = \underset{{\ell^{\infty}(\Gamma)}}{*} (\cA_{e}, E_{e})_{e \in E(\Gamma)}.
 $$
\end{defn}

\begin{rem} If $\Gamma$ is an infinite graph with a weighting that is not $\ell^{1}$, then we can still define $M(\Gamma)$ as in \ref{defn:vNGraph} although it will be a semifinite algebra.  Given $e \in E(\Gamma)$ connecting vertices $v$ and $w$, the compressed algebra $(p_{v} + p_{w})\cA_{e}(p_{v} + p_{w})$ is still finite, and if $E_{e}: \cA_{e} \rightarrow \ell^{\infty}(\Gamma)$ is the (tracial-weight) preserving conditional expectation, then $E_{e}$ is clearly normal on $(p_{v} + p_{w})\cA_{e}(p_{v} + p_{w})$ and is the identity on $(1 - p_{v} - p_{w})\cA_{e}(1 - p_{v} - p_{w})$.  Therefore one can take the algebraic free product $Q$ of $(\cA_{e})_{e\in E(\Gamma)}$ with amalgamation over $\ell^{\I}(\Gamma)$ and represent it on $L^{2}(Q, Tr \circ \underset{\ell^{\infty}(\Gamma)}{*}E_{e})$ to obtain $\cM(\Gamma)$.

 \end{rem}

 \begin{defn} \label{defn:H2} Let $v, w \in V(\Gamma)$ We write $v \sim w$ if $v$ and $w$ are connected by at least 1 edge in $\Gamma$ and denote $n_{v, w}$ be the number of edges joining $v$ and $w$.  We set $\alpha^{\Gamma}_{v} = \sum_{w\sim v} n_{v, w}\gamma_{w}$, and define $B(\Gamma) = \{ v \in V(\Gamma) : \gamma_{v} > \alpha^{\Gamma}_{v}\}$.  \end{defn}

 For the rest of this section, we assume $\Gamma$ is finite. We show that $\cM(\Gamma)$ is the direct sum of an interpolated free group factor and a finite dimensional abelian algebra.  More precisely, we prove the following theorem:

\begin{thm} \label{thm:H1}

Let $\Gamma$ and $\Gamma'$ be connected, finite, loopless, and weighted graphs with at least 2 edges.  Then $\cM(\Gamma) \cong \overset{p^{\Gamma}}{L(\F_{t_{\Gamma}})} \oplus \underset{{v \in B(\Gamma)}}{\bigoplus} \overset{r_{v}^{\Gamma}}{\underset{\gamma_{v} - \alpha^{\Gamma}_{v}}{\C}}$ where $r_{v}^{\Gamma} \leq p_{v}$ and $t_{\Gamma}$ is such that this algebra has the appropriate free dimension.  Furthermore, if $\Gamma$ is a subgraph of $\Gamma'$, then $p^{\Gamma}\cM(\Gamma)p^{\Gamma} \overset{s.e.}{\hookrightarrow} p^{\Gamma}\cM(\Gamma')p^{\Gamma}$.

\end{thm}

 Notice that since we are assuming that all vertices have positive weight, it follows that $p_{v}p^{\Gamma}\neq 0$ for all $v \in \Gamma$.  We prove Theorem \ref{thm:H1} in a series of lemmas.

\begin{lem} \label{lem:H1}

Let $\Gamma$ be a finite, connected, weighted, loopless graph with 2 edges.  Then $M(\Gamma)$ is of the form in Theorem \ref{thm:H1}.

\end{lem}

\begin{proof}

 Set $D = \ell^{\infty}(\Gamma)$.  There are two overlying cases to consider.  One where $\Gamma$ has 2 vertices and the other where $\Gamma$ has 3 vertices.

   \emph{\underline{Case 1}}: Assume that $\Gamma$ has 2 vertices $v, w$ and 2 edges $e_{1}$ and $e_{2}$ connecting $v$ and $w$ and without loss of generality assume $\gamma_{v} \geq \gamma_{w}$.  We obtain the desired formula for $\cM(\Gamma)$ by examining the following sequence of inclusions:

 \begin{align*}
 \cN_{0} &= \left(\overset{p_{w}}{\underset{\gamma_{w}}{\C }}\oplus \overset{p_{v}^{I}}{\underset{\gamma_{w}}{\C}} \oplus \underset{\gamma_{v} - \gamma_{w}}{\C}\right) \underset{D}* \left(\overset{p_{w}}{\underset{\gamma_{w}}{\C }}\oplus \overset{q_{v}^{I}}{\underset{\gamma_{w}}{\C}} \oplus \underset{\gamma_{v} - \gamma_{w}}{\C}\right)\\
 \cap&\\
     \cN_{1} &= \left(\overset{p_{w}, p_{v}^{I}}{\underset{\gamma_{w}}{M_{2}}} \oplus \underset{\gamma_{v} - \gamma_{w}}{\C}\right) \underset{D}{*} \left(\overset{p_{w}}{\underset{\gamma_{w}}{\C }}\oplus \overset{q_{v}^{I}}{\underset{\gamma_{w}}{\C}} \oplus \underset{\gamma_{v} - \gamma_{w}}{\C}\right)\\
      \cap&\\
       \cN_{2} &= \left(\overset{p_{w}, p_{v}^{I}}{\underset{\gamma_{w}}{M_{2}}} \oplus \underset{\gamma_{v} - \gamma_{w}}{\C}\right) \underset{D}{*} \left(\overset{p_{w}, q_{v}^{I}}{\underset{\gamma_{w}}{M_{2}}} \oplus \underset{\gamma_{v} - \gamma_{w}}{\C}\right)\\
        \cap&\\
        \cN_{3} &= \left(\underset{2\gamma_{w}}{M_{2} \otimes L(\Z)} \oplus \underset{\gamma_{v} - \gamma_{w}}{\C}\right) \underset{D}{*} \left(\overset{p_{w}}{\underset{\gamma_{w}}{M_{2}}} \oplus \underset{\gamma_{v} - \gamma_{w}}{\C}\right)\\
         \cap&\\
          \cM(\Gamma) &= \left(\underset{2\gamma_{w}}{M_{2} \otimes L(\Z)} \oplus \underset{\gamma_{v} - \gamma_{w}}{\C}\right) \underset{D}{*} \left(\underset{2\gamma_{w}}{M_{2} \otimes L(\Z)} \oplus \underset{\gamma_{v} - \gamma_{w}}{\C}\right),
          \end{align*}
        where $p_{v}$ decomposes as $(1 \otimes e_{2, 2}) \oplus 1$ in $\cA_{e_{1}}$ and $\cA_{e_{2}}$ with $p_{v}^{I} = 1 \otimes e_{2,2}$ in $\cA_{e_{1}}$ and $q_{v}^{I} = 1 \otimes e_{2,2}$ in $\cA_{e_{2}}$.   From Lemma \ref{lem:DR1} and \cite{MR1201693}, we see that
        $$
        p_{v}\cN_{0}p_{v} = \overset{p_{v}^{I}}{\underset{\frac{\gamma_{w}}{\gamma{v}}}{\C}} \oplus \underset{\frac{\gamma_{v} - \gamma_{w}}{\gamma{v}}}{\C}*\overset{q_{v}^{I}}{\underset{\frac{\gamma_{w}}{\gamma{v}}}{\C}} \oplus \underset{\frac{\gamma_{v} - \gamma_{w}}{\gamma{v}}}{\C} = \begin{cases} \underset{2\frac{\gamma_{v} - \gamma_{w}}{\gamma_{v}}}{M_{2} \otimes L(\Z)} \oplus \overset{p_{v}^{I} \wedge q_{v}^{I}}{\underset{\frac{2\gamma_{w} - \gamma_{v}}{\gamma_{v}}}{\C}} & \textrm{if } 2\gamma_{w} \geq\gamma_{v}\\ \underset{\frac{2\gamma_{w}}{\gamma_{v}}}{M_{2} \otimes L(\Z)} \oplus \overset{(p_{v} - p_{v}^{I})\wedge(p_{v} - q_{v}^{I})}{\underset{\frac{\gamma_{v} - 2\gamma_{w}}{\gamma_{v}}}{\C}} & \textrm{if }\gamma_{v} > 2\gamma_{w} \end{cases}
        $$
        where in the first algebra, the identity element copy of $\C$ is $p_{v}^{I} \wedge q_{v}^{I}$ and and in the second algebra, the identity of the copy of $\C$ is orthogonal to both $p_{v}^{I}$ and $q_{v}^{I}$.

     \emph{\underline{Case 1a}}:  Assume $2\gamma_{w} \geq \gamma_{v}$.  As $p_{v}\wedge q_{v}$ is minimal and central in $\cN_{0}$, we see that
      $$
      \cN_{1} = \underset{3(\gamma_{v} - \gamma_{w})}{M_{3}\otimes L(\Z)} \oplus \underset{2\gamma_{w} - \gamma_{v}}{\overset{p_{v}^{I} \wedge q_{v}^{I}}{M_{2}}}.
      $$
       By \cite{MR1201693}, the projections $p_{v}^{I}$ and $q_{v}^{I}$ are equivalent in $\cN_{0}$, so it follows that $p_{w}$ is equivalent to $q_{v}^{I}$ in $\cN_{1}$.  Therefore by Lemma \ref{lem:DR2},
        $$
        p_{w}\cN_{2}p_{w} = p_{w}\cN_{1}p_{w} * L(\Z) = (\underset{\frac{\gamma_{v} - \gamma_{w}}{\gamma_{w}}}{L(\Z)} \oplus \underset{\frac{2\gamma_{w} - \gamma_{v}}{\gamma_{w}}}{\C}) * L(\Z),
        $$
         which is an interpolated free group factor $L(\F_{t})$ \cite{MR1201693}.  As the central support of $p_{w}$ in $\cN_{2}$ is 1, it follows that $\cN_{2}$ is also an interpolated free group factor $L(\F_{t_{1}})$.  To finish up this case, we write
         \begin{align*}
         \cN_{2} \subset \cN_{3} &= \left(\underset{2\gamma_{w}}{M_{2} \otimes L(\Z)} \oplus \underset{\gamma_{v} - \gamma_{w}}{\C}\right) \underset{D}{*} \left(\overset{p_{w}}{\underset{\gamma_{w}}{M_{2}}} \oplus \underset{\gamma_{v} - \gamma_{w}}{\C}\right)  \textrm{ and }\\
      \cN_{3} \subset \cM(\Gamma) &= \left(\underset{2\gamma_{w}}{M_{2} \otimes L(\Z)} \oplus \underset{\gamma_{v} - \gamma_{w}}{\C}\right) \underset{D}{*} \left(\underset{2\gamma_{w}}{M_{2} \otimes L(\Z)} \oplus \underset{\gamma_{v} - \gamma_{w}}{\C}\right),
      \end{align*}
       and use Lemma \ref{lem:DR1} twice, as well as the amplification formula for interpolated free group factors to obtain that $\cM(\Gamma)$ is an interpolated free group factor.

      \emph{\underline{Case 1b}}: The case $\gamma_{v} > 2\gamma_{w}$ for $\cN_{0}$ is treated exactly the same as the first with only the caveat that the central support of $p_{w}$ in $\cN_{1}$ is a projection of trace $3\gamma_{w}$, so $\cN_{1}$, and thus $\cN_{2}$, $\cN_{3}$, and $\cM(\Gamma)$, have a minimal central projection of trace $\gamma_{v} - 2\gamma_{w}$.\\

 \emph{\underline{Case 2}}: Next we consider the case where $\Gamma$ has 3 vertices $v_{1}$, $v_{2}$, and $v_{3}$ with weights $\gamma_{1}$, $\gamma_{2}$, and $\gamma_{3}$ respectively, where $v_{2}$ is connected to $v_{1}$ by $e_{1}$ and to $v_{3}$ by $e_{2}$.  There are two sub-cases to consider.  The first is when $\gamma_{2} \geq \gamma_{1} \geq \gamma_{3}$, and the second is when $\gamma_{1} > \gamma_{2}$ and $\gamma_{1} \geq \gamma_{3}$.

 \emph{\underline{Case 2a}}: We examine the following sequence of inclusions:
  \begin{align*}
  \cN_{0} &= \left(\overset{p_{v_{1}}}{\underset{\gamma_{1}}{\C}} \oplus \overset{p^{I}_{2}}{\underset{\gamma_{1}}{\C}} \oplus \overset{p^{II}_{2}}{\underset{\gamma_{2} - \gamma_{1}}{\C}} \oplus \overset{p_{v_{3}}}{\underset{\gamma_{3}}{\C}}\right) \underset{D}{*} \left(\overset{p_{v_{1}}}{\underset{\gamma_{1}}{\C}} \oplus \overset{q^{I}_{2}}{\underset{\gamma_{2} - \gamma_{3}}{\C}} \oplus \overset{q^{II}_{2}}{\underset{\gamma_{3}}{\C}} \oplus \overset{p_{v_{3}}}{\underset{\gamma_{3}}{\C}}\right)\\
  \cap& \\
  \cN_{1} &= \left(\overset{p_{v_{1}}, p_{2}^{I}}{\underset{\gamma_{1}}{M_{2}}} \oplus \overset{p_{2}^{II}}{\underset{\gamma_{2} - \gamma_{1}}{\C}} \oplus \overset{p_{v_{3}}}{\underset{\gamma_{3}}{\C}}\right) \underset{D}{*} \left(\overset{p_{v_{1}}}{\underset{\gamma_{1}}{\C}} \oplus \overset{q_{2}^{I}}{\underset{\gamma_{2} - \gamma_{3}}{\C}} \oplus \overset{q_{2}^{II},p_{v_{3}}}{\underset{\gamma_{3}}{M_{2}}}\right)\\
  \cap &\\
  \cN_{2} &= \left(\underset{2\gamma_{1}}{M_{2}\otimes L(\Z)} \oplus \overset{p_{2}^{II}}{\underset{\gamma_{2} - \gamma_{1}}{\C}} \oplus \overset{p_{v_{3}}}{\underset{\gamma_{3}}{\C}}\right) \underset{D}{*} \left(\overset{p_{v_{1}}}{\underset{\gamma_{1}}{\C}} \oplus \overset{q_{2}^{I}}{\underset{\gamma_{2} - \gamma_{3}}{\C}} \oplus \overset{q_{2}^{II}, p_{v_{3}}}{\underset{\gamma_{3}}{M_{2}}}\right)\\
  \cap &\\
   \cM(\Gamma) &= \left(\underset{\gamma_{1}}{M_{2}\otimes L(\Z)} \oplus \overset{p_{2}^{II}}{\underset{\gamma_{2} - \gamma_{1}}{\C}} \oplus \overset{p_{v_{3}}}{\underset{\gamma_{3}}{\C}}\right) \underset{D}{*} \left(\overset{p_{v_{1}}}{\underset{\gamma_{1}}{\C}} \oplus \overset{q_{2}^{I}}{\underset{\gamma_{2} - \gamma_{3}}{\C}} \oplus \underset{2\gamma_{3}}{M_{2}\otimes L(\Z)}\right),
  \end{align*}
  where $p_{v_{2}}$ decomposes as $1 \otimes e_{22} \oplus 1 \oplus 0$ in $\cA_{e_{1}}$ and $0 \oplus 1 \oplus 1 \otimes e_{1, 1}$ in $\cA_{e_{2}}$.  We set $p_{2}^{I}$ and $p_{2}^{II}$ as the summands of $p_{v_{2}}$ supported in the diffuse and atomic parts of $\cA_{e_{1}}$ respectively and $q_{2}^{I}$ and $q_{2}^{II}$ as the summands of $p_{v_{2}}$ supported in the atomic and diffuse parts of $\cA_{e_{2}}$ respectively.  As above,
    $$
    p_{v_{2}}\cN_{0}p_{v_{2}} = \overset{p^{I}_{2}}{\underset{\frac{\gamma_{1}}{\gamma_{2}}}{\C}} \oplus \overset{p^{II}_{2}}{\underset{\frac{\gamma_{2} - \gamma_{1}}{\gamma_{2}}}{\C}} * \overset{q^{I}_{2}}{\underset{\frac{\gamma_{2} - \gamma_{3}}{\gamma_{2}}}{\C}} \oplus \overset{q^{II}_{2}}{\underset{\frac{\gamma_{3}}{\gamma_{2}}}{\C}} = \begin{cases} \underset{2\frac{\gamma_{2} - \gamma_{1}}{\gamma_{2}}}{M_{2} \otimes L(\Z)} \oplus \overset{p_{2}^{I} \wedge q_{2}^{I}}{\underset{\frac{\gamma_{1} - \gamma_{3}}{\gamma_{2}}}{\C}} \oplus \overset{p_{2}^{I} \wedge q_{2}^{II}}{\underset{\frac{\gamma_{1} - \gamma_{2} +  \gamma_{3}}{\gamma_{2}}}{\C}} & \textrm{ if } \gamma_{2} \leq \gamma_{1} + \gamma_{3} \\ \underset{2\frac{\gamma_{3}}{\gamma_{2}}}{M_{2} \otimes L(\Z)} \oplus \overset{p_{2}^{I} \wedge q_{2}^{I}}{\underset{\frac{\gamma_{1} - \gamma_{3}}{\gamma_{2}}}{\C}} \oplus \overset{p_{2}^{II} \wedge q_{2}^{I}}{\underset{\frac{\gamma_{2} - \gamma_{1} -  \gamma_{3}}{\gamma_{2}}}{\C}} & \textrm{ if } \gamma_{2} > \gamma_{1} + \gamma_{3} \end{cases}.
    $$

 \emph{\underline{Case 2a(i)}}: Assume $\gamma_{2} \leq \gamma_{1} + \gamma_{3}$.  Since the two new matrix units in $\cN_{1}$ introduce equivalences between $p_{v_{1}}$ and $p_{2}^{I}$ and between $q_{2}^{II}$ and $p_{v_{3}}$ respectively, we see that $\cN_{1}$ has the same number of summands as $p_{v_{2}}\cN_{0}p_{v_{2}}$, but with suitable amplifications.  Explicitly, we find that
   $$
   \cN_{1} = \underset{4(\gamma_{2} - \gamma_{1})}{M_{4}\otimes L(\Z)} \oplus \overset{p_{2}^{I}\wedge q_{2}^{I}}{\underset{\gamma_{1} - \gamma_{3}}{M_{2}}} \oplus \overset{p_{2}^{I} \wedge q_{2}^{II}}{\underset{\gamma_{1} + \gamma_{3} - \gamma_{2}}{M_{3}}}
   $$
    where the central support of $p_{v_{1}}$ is the identity.  By applying Lemma \ref{lem:DR1} and applying the same reasoning as case 1,
    we see that $\cN_{2}$ is an interpolated free group factor.  Applying Lemma \ref{lem:DR1} again shows that $\cM(\Gamma)$ is an interpolated free group factor.

 \emph{\underline{Case 2a(ii)}}: Assume $\gamma_{2} > \gamma_{1} + \gamma_{3}$.  This case is treated in the same way as above except that in $\cN_{1}$, $q_{2}^{I} \wedge p_{2}^{II}$ with trace $\gamma_{2} - \gamma_{3} - \gamma_{1}$ is minimal and central, so it is minimal and central in $\cM(\Gamma)$.

 \emph{\underline{Case 2b}}: Now let $\gamma_{1}$ be the largest weight. First assume $\gamma_{3} \geq \gamma_{2}$.  We consider the algebra
 $$
 \cN_{1} = \left(\overset{p_{1}^{I}}{\underset{\gamma_{1} - \gamma_{2}}{\C}}  \oplus \overset{p_{v_{2}}}{\underset{\gamma_{2}}{M_{2}}} \oplus \overset{p_{v_{3}}}{\underset{\gamma_{3}}{\C}}\right) \underset{D}{*} \left(\overset{p_{v_{1}}}{\underset{\gamma_{1}}{\C}} \oplus \overset{p_{v_{2}}}{\underset{\gamma_{2}}{M_{2}}} \oplus \overset{p_{3}^{I}}{\underset{\gamma_{3} - \gamma_{2}}{\C}}\right),
 $$
 where the projections orthogonal to $p_{v_{2}}$ in each copy of $M_{2}$ sit under $p_{i}$ and $p_{i}^{I} \leq p_{v_{i}}$ for $i = 1$ or 3.  It follows that $\cN_{1} = \overset{p_{1}^{I}}{\underset{\gamma_{1} - \gamma_{2}}{\C}} \oplus \overset{p_{v_{2}}}{\underset{\gamma_{2}}{M_{3}}} \oplus \overset{p_{3}^{I}}{\underset{\gamma_{3} - \gamma_{2}}{\C}}$, so tensoring each copy of $M_{2}$ with $L(\Z)$ and using the standard arguments as above show that
 $$
 \cM(\Gamma)  = \overset{p_{1}^{I}}{\underset{\gamma_{1} - \gamma_{2}}{\C}} \oplus \underset{3\gamma_{2}}{L(\F_{t})} \oplus \overset{p_{3}^{I}}{\underset{\gamma_{3} - \gamma_{2}}{\C}}.
 $$
 Finally, if $\gamma_{2} > \gamma_{3}$ then we consider
 $$
 \cN_{1} = \left(\overset{p_{1}^{I}}{\underset{\gamma_{1} - \gamma_{2}}{\C}}  \oplus \overset{p_{v_{2}}}{\underset{\gamma_{2}}{M_{2}}} \oplus \overset{p_{v_{3}}}{\underset{\gamma_{3}}{\C}}\right) \underset{D}{*} \left(\overset{p_{v_{1}}}{\underset{\gamma_{1}}{\C}} \oplus \overset{p_{2}^{I}}{\underset{\gamma_{2} - \gamma_{3}}{\C}} \oplus \overset{p_{v_{3}}}{\underset{\gamma_{3}}{M_{2}}}\right) = \overset{p_{1}^{I}}{\underset{\gamma_{1} - \gamma_{2}}{\C}} \oplus \underset{\gamma_{3}}{M_{3}} \oplus \underset{\gamma_{2} - \gamma_{3}}{M_{2}},
 $$
 where the central support of $p_{v_{2}}$ is $1 - p_{1}^{I}$.  Therefore, tensoring each copy of $M_{2}$ with $L(\Z)$ gives $\cM(\Gamma) = \overset{p_{1}^{I}}{\underset{\gamma_{1} - \gamma_{2}}{\C}} \oplus \underset{2\gamma_{2} + \gamma_{3}}{L(\F_{t})}$ as desired.
\end{proof}

We now inductively assume that for some $\Gamma$, $\cM(\Gamma)$ has the form as described in Theorem \ref{thm:H1}.\\

\begin{lem} \label{lem:H2}

  Suppose $\Gamma'$ is a graph obtained from $\Gamma$ by adding an edge $e$ connecting two vertices $v$ and $w$ of $\Gamma$ (so that in particular $\Gamma$ and $\Gamma'$ have the same underlying set of vertices with the same weighting).  Assume that
  $$
  \cM(\Gamma) = \overset{p^{\Gamma}}{L(\F_{t_{\Gamma}})} \oplus \underset{{v \in B(\Gamma)}}{\bigoplus} \overset{r_{v}^{\Gamma}}{\underset{\gamma_{v} - \alpha^{\Gamma}_{v}}{\C}}
  $$
  as in Theorem \ref{thm:H1}.  Then
  $$
  \cM(\Gamma') = \overset{p^{\Gamma'}}{L(\F_{t_{\Gamma'}})} \oplus \underset{{v \in B(\Gamma')}}{\bigoplus} \overset{r_{v}^{\Gamma'}}{\underset{\gamma_{v} - \alpha^{\Gamma'}_{v}}{\C}}
  $$
  where $p^{\Gamma} \leq p^{\Gamma'}$, $r_{v}^{\Gamma'} \leq r_{v}^{\Gamma}$, and $p^{\Gamma}\cM(\Gamma)p^{\Gamma} \overset{s.e.}{\hookrightarrow} p_{\Gamma}\cM(\Gamma')p_{\Gamma}$.

\end{lem}

\begin{proof}

  We use the convention that if the term $\overset{p}{\underset{\alpha}{\C}}$ appears where $\alpha \leq 0$ then this term is identically zero.  All parts of the proof below are valid if this modification is made.

  Set $D = \ell^{\infty}(\Gamma') = \ell^{\infty}(\Gamma)$ and let the new edge $e$ connect $v$ to $w$ with $\gamma_{v} \geq \gamma_{w}$.  We examine the following sequence of inclusions:

\begin{align*}
\cM(\Gamma) \subset \cN_{1} &= \cM(\Gamma) \underset{D}{*} \left(\overset{p_{w}}{\underset{\gamma_{w}}{\C}} \oplus \left(\bigoplus_{k=1}^{n} \overset{p_{v, k}}{\underset{\gamma_{w}/n}{\C}} \oplus \overset{p_{v}^{I}}{\underset{\gamma_{v} - \gamma_{w}}{\C}}\right) \oplus \ell^{\I}(\Gamma \setminus \{v, w\})\right)\\
\cap &\\
\cN_{2} &= \cM(\Gamma) \underset{D}{*} \left(\bigoplus_{k=1}^{n} \overset{p_{w, k}}{\underset{\gamma_{w}/n}{\C}} \oplus  \left(\bigoplus_{k=1}^{n} \overset{p_{v, k}}{\underset{\gamma_{w}/n}{\C}} \oplus \overset{p_{v}^{I}}{\underset{\gamma_{v} - \gamma_{w}}{\C}}\right) \oplus \ell^{\I}(\Gamma \setminus \{v, w\})\right)\\
\cap &\\
\cN_{3} &= \cM(\Gamma) \underset{D}{*} \left(\bigoplus_{k=1}^{n}\overset{p_{w, k}, p_{v, k}}{\underset{\gamma_{w}/n}{M_{2}}} \oplus \overset{p_{v}^{I}}{\underset{\gamma_{v} - \gamma_{w}}{\C}} \oplus \ell^{\I}(\Gamma \setminus \{v, w\})\right)\\
\cap &\\
\cM(\Gamma') &= \cM(\Gamma) \underset{D}{*} \left(L(\Z) \otimes M_{2} \oplus \overset{p_{v}^{I}}{\underset{\gamma_{v} - \gamma_{w}}{\C}} \oplus \ell^{\I}(\Gamma \setminus \{v, w\})\right).
\end{align*}
The projections $p_{w, k}$ are an orthogonal family with trace $\gamma_{w}/n$ in $\cA_{e}$ whose sum is $p_{w}$.  The projection $p_{v}$ decomposes as $\sum_{k=1}^{n}p_{v, k} + p_{v}^{I}$ with  $p_{v}^{I}$ supported in the atomic part of $\cA_{e}$ and the $p_{v, k}$ are an orthogonal family of projections with trace $\gamma_{w}/n$ supported in the diffuse part of $\cA_{e}$. The positive integer $n$ is chosen large enough such that $\frac{1}{n} + \frac{\gamma_{w} - \alpha^{\Gamma}_{w}}{\gamma_{w}} < 1$ and $\frac{\gamma_{w}}{n\gamma_{v}} + \frac{\gamma_{v} - \alpha^{\Gamma}_{v}}{\gamma_{v}} < 1$.  From the induction hypothesis,
$$
p_{v}\cM(\Gamma)p_{v} = \overset{p^{\Gamma}_{v}}{L(\F_{t_{v}})} \oplus \overset{r_{v}^{\Gamma}}{\underset{\frac{\gamma_{v} - \alpha_{v}}{\gamma_{v}}}{\C}}, \textrm{ and }p_{w}\cM(\Gamma)p_{w} = \overset{p^{\Gamma}_{w}}{L(\F_{t_{w}})} \oplus \overset{r_{w}^{\Gamma}}{\underset{\frac{\gamma_{w} - \alpha_{w}}{\gamma_{w}}}{\C}},
$$
with $p^{\Gamma}_{u} = p^{\Gamma}p_{u}$ for any vertex $u$.   From Lemma \ref{lem:DR1},
$$
p_{v}\cN_{1}p_{v} = \left(\overset{p^{\Gamma}_{v}}{L(\F_{t_{v}})} \oplus \overset{r_{v}^{\Gamma}}{\underset{\frac{\gamma_{v} - \alpha^{\Gamma}_{v}}{\gamma_{v}}}{\C}}\right) * \left(\bigoplus_{k=1}^{n} \overset{p_{v, k}}{\underset{\frac{\gamma_{w}}{n\gamma_{v}}}{\C}} \oplus \overset{p_{v}^{I}}{\underset{\frac{\gamma_{v} - \gamma_{w}}{\gamma_{v}}}{\C}}\right) = L(\F_{t_{v,1}}) \oplus \overset{p_{v}^{I}\wedge r_{v}^{\Gamma}}{\underset{\frac{\gamma_{v} - \alpha^{\Gamma'}_{v}}{\gamma_{v}}}{\C}}.
$$
 Lemma \ref{lem:DR1} applied to the inclusion
 $$
 \left(\overset{p^{\Gamma}_{v}}{\C} \oplus \overset{r_{v}^{\Gamma}}{\underset{\frac{\gamma_{v} - \alpha_{v}}{\gamma_{v}}}{\C}}\right) * \left(\bigoplus_{k=1}^{n} \overset{p_{v, k}}{\underset{\frac{\gamma_{w}}{n\gamma_{v}}}{\C}} \oplus \overset{p_{v}^{I}}{\underset{\frac{\gamma_{v} - \gamma_{w}}{\gamma_{v}}}{\C}}\right) \rightarrow \left(\overset{p^{\Gamma}_{v}}{L(\F_{t_{v}})} \oplus \overset{r_{v}^{\Gamma}}{\underset{\frac{\gamma_{v} - \alpha_{v}}{\gamma_{v}}}{\C}}\right) * \left(\bigoplus_{k=1}^{n} \overset{p_{v, k}}{\underset{\frac{\gamma_{w}}{n\gamma_{v}}}{\C}} \oplus \overset{p_{v}^{I}}{\underset{\frac{\gamma_{v} - \gamma_{w}}{\gamma_{v}}}{\C}}\right),
 $$
  shows that the inclusion $\displaystyle L(\F_{t_{v}}) = p^{\Gamma}_{v}\cM(\Gamma)p^{\Gamma}_{v} \rightarrow p^{\Gamma}_{v}\cN_{1}p^{\Gamma}_{v}$
  is equivalent to the canonical inclusion
  $$
   L(\F_{t_{v}}) \rightarrow L(\F_{t_{v}}) * p^{\Gamma}_{v}\left[\left(\overset{p'_{v}}{\C} \oplus \underset{\frac{\gamma_{v} - \alpha_{v}}{\gamma_{v}}}{\C}\right) * \left( \bigoplus_{k=1}^{n} \overset{p_{v, k}}{\underset{\frac{\gamma_{w}}{n\gamma_{v}}}{\C}} \oplus \overset{p_{v}^{I}}{\underset{\frac{\gamma_{v} - \gamma_{w}}{\gamma_{v}}}{\C}}\right)\right]p^{\Gamma}_{v}
   $$
so $p^{\Gamma}_{v}\cM(\Gamma)p^{\Gamma}_{v}\overset{s.e.}{\hookrightarrow} p^{\Gamma}_{v}\cN_{1}p^{\Gamma}_{v}$.  From Remark \ref{rem:Dyk}, $p^{\Gamma}\cM(\Gamma)p^{\Gamma} \overset{s.e.}{\hookrightarrow} p^{\Gamma}\cN_{1}p^{\Gamma}$ as well.

  By Lemma \ref{lem:DR1} we have
$$
p_{w}\cN_{2}p_{w} = L(\Z/n\Z) * p_{w}\cN_{1}p_{w} = L(\Z/n\Z) * \left(\overset{p_{w}^{\cN_{1}}}{L(\F_{t_{w}})} \oplus \overset{r_{w}^{\cN_{1}}}{\underset{\gamma_{w}}{\C}}\right),
$$
 where $p_{w}^{\cN_{1}} = p_{w}p^{\cN_{1}}$ with $p^{\cN_{1}}$ the central support of $p_{\Gamma}$ in $\cN_{1}$ (note $p_{w}^{\cN_{1}} \geq p_{w}^{\Gamma}$ so $r_{w}^{\cN_{1}} \leq r_{w}^{\Gamma}$ which implies $\delta_{w} \leq \gamma_{w} - \alpha^{\Gamma}_{w}$).  From these observations, it follows that $p_{w}\cN_{2}p_{w}$ is an interpolated free group factor (since $n$ was chosen such that $\frac{\gamma_{w}}{n\gamma_{v}} + \frac{\gamma_{v} - \alpha_{v}}{\gamma_{v}} < 1$) and the arguments used in the inclusion $\cM(\Gamma) \rightarrow \cN_{1}$ imply $p_{w}^{\cN_{1}}\cN_{1}p_{w}^{\cN_{1}} \overset{s.e.}{\hookrightarrow} p_{w}^{\cN_{1}}\cN_{2}p_{w}^{\cN_{1}}$. Therefore $p_{w}^{\Gamma}\cN_{1}p_{w}^{\Gamma} \overset{s.e.}{\hookrightarrow} p_{w}^{\Gamma}\cN_{2}p_{w}^{\Gamma}$ so $p^{\Gamma}\cN_{1}p^{\Gamma} \overset{s.e.}{\hookrightarrow} p^{\Gamma}\cN_{2}p^{\Gamma}$.  Also, observe that since the projections $p_{v, k}$ and $p_{w, k}$ lie in the interpolated free group factor summand of $\cN_{2}$, they are equivalent in $\cN_{2}$. We now define algebras $\cN_{2, j}$ for $j = 0,...,n$ so that

$$
\cN_{2} = \cN_{2, 0} \subset \cN_{2, 1} \subset \cN_{2, 2} \subset \cdots \subset \cN_{2, n} = \cN_{3} \textrm{ where }
$$ $$
\cN_{2, j} = \left(\bigoplus_{k=j+1}^{n}  \overset{p_{w, k}}{\underset{\gamma_{w}/n}{\C}} \oplus \bigoplus_{k=1}^{j} \overset{p_{w, k}, p_{v, k}}{\underset{\gamma_{w}/n}{M_{2}}} \oplus \bigoplus_{k=j+1}^{n}\overset{p_{v, k}}{\underset{\gamma_{w}/n}{\C}} \oplus \overset{p_{v}^{I}}{\underset{\gamma_{v} - \gamma_{w}}{\C}} \oplus \ell^{\I}(\Gamma \setminus \{v, w\})\right) \underset{D}{*} \cM(\Gamma).
$$

Let $p^{\cN_{2}}$ be the central support of $p_\Gamma$ in $\cN_{2}$.  Applying Lemma \ref{lem:DR2} to the inclusion
$$
p_{w, j+1}\cN_{2, j}p_{w, j+1} \rightarrow p_{w, j+1}\cN_{2,j+1}p_{w, j+1} = p_{w, j+1}\cN_{2, j}p_{w, j+1} * L(\Z)
$$
shows that this inclusion is a standard embedding, so it follows from Remark \ref{rem:Dyk} that $p^{\cN_{2}}\cN_{2, j}p^{\cN_{2}} \overset{s.e.}{\hookrightarrow} p^{\cN_{2}}\cN_{2, j+1}p^{\cN_{2}}$, implying $p^{\Gamma}\cN_{2, j}p^{\Gamma} \overset{s.e.}{\hookrightarrow} p^{\Gamma}\cN_{2, j+1}p^{\Gamma}$ for all $j$.  Inductively,
$$
\cN_{3} = \left(\overset{p^{\cN_{2}}}{L(\F_{t_{3}})} \oplus \overset{p_{v}^{I}\wedge r_{v}^{\Gamma}}{\underset{\gamma_{v} - \alpha_{v}^{\Gamma'}}{\C}} \bigoplus_{u \in L(\Gamma)\setminus\{v, w\}} \underset{\gamma_{u} - \alpha^{\Gamma'}_{u}}{\overset{r^{\Gamma}_{u}}{\C}}\right)
$$
and $p^{\Gamma}\cN_{2}p^{\Gamma} \overset{s.e.}{\hookrightarrow} p^{\Gamma}\cN_{3}p^{\Gamma}$.  To finish, we look at the sequence of algebras
$$
\cN_{3} = \cN_{3, 0} \subset \cN_{3, 1} \subset ... \subset \cN_{3, n} = \cM(\Gamma') \textrm{ where }
$$ $$\cN_{3, j} = \left(\bigoplus_{k=1}^{j} \overset{p_{w, k} + q_{w, k}}{M_{2} \otimes L(\Z)} \oplus \bigoplus_{k=j+1}^{n} \overset{p_{w, k}, p_{v, k}}{M_{2}} \oplus \overset{p_{v}^{I}}{\underset{\gamma_{v} - \gamma_{w}}{\C}} \bigoplus \ell^{\infty}(\Gamma \setminus \{v, w\})\right) \underset{D}{*} \cM(\Gamma).
$$
Lemma \ref{lem:DR1} implies that the inclusion
$$
p_{w, j+1}\cN_{3, j}p_{w, j+1} \rightarrow p_{w, j+1}\cN_{3, j+1}p_{w, j+1} = p_{w, j+1}\cN_{3, j}p_{w, j+1} * L(\Z)
$$
is a standard embedding, so by Remark \ref{rem:Dyk}, $p^{\cN_{2}}\cN_{3, j}p^{\cN_{2}} \overset{s.e.}{\hookrightarrow} p^{\cN_{2}}\cN_{3, j+1}p^{\cN_{2}}$ and thus $p^{\Gamma}\cN_{3, j}p^{\Gamma} \overset{s.e.}{\hookrightarrow} p^{\Gamma}\cN_{3, j+1}p^{\Gamma}$.  Therefore the inclusion $p^{\Gamma}\cN_{3}p^{\Gamma} \rightarrow p^{\Gamma}\cM(\Gamma')p^{\Gamma}$ is standard since it is a composite of standard embeddings.  This implies $\cM(\Gamma')$ has the desired formula and $p^{\Gamma}\cM(\Gamma)p^{\Gamma} \overset{s.e.}{\hookrightarrow} p^{\Gamma}\cM(\Gamma')p^{\Gamma}$.\end{proof}
We again assume that $\cM(\Gamma)$ is in the form of Theorem \ref{thm:H1}.

\begin{lem} \label{lem:H3}

 Let $\Gamma'$ be a weighted graph obtained from $\Gamma$ by adding a vertex $v$ and an edge $e$ connecting $v$ to $w \in V(\Gamma)$ woth weighting $\gamma_{v}$, and assume $\cM(\Gamma) = \overset{p^{\Gamma}}{L(\F_{t_{\Gamma}})} \oplus \underset{{v \in B(\Gamma)}}{\bigoplus} \overset{r_{v}^{\Gamma}}{\underset{\gamma_{v} - \alpha^{\Gamma}_{v}}{\C}}$ with notation as in Theorem \ref{thm:H1}.  Then
 $$
 \cM(\Gamma') = \overset{p^{\Gamma'}}{L(\F_{t_{\Gamma'}})} \oplus \underset{{u \in B(\Gamma')}}{\bigoplus} \overset{r_{u}^{\Gamma'}}{\underset{\gamma_{u} - \alpha^{\Gamma'}_{u}}{\C}}
 $$ where $p^{\Gamma} \leq p^{\Gamma'}$, $r_{u}^{\Gamma'} \leq r_{u}^{\Gamma}$ for all $u$, and $p_{\Gamma}\cM(\Gamma)p_{\Gamma} \overset{s.e.}{\hookrightarrow} p_{\Gamma}\cM(\Gamma')p_{\Gamma}$.
\end{lem}

Notice that the natural inclusion $\cM(\Gamma) \rightarrow \cM(\Gamma')$ is not unital, but the compressed inclusion $p_{\Gamma}\cM(\Gamma)p_{\Gamma} \rightarrow p_{\Gamma}\cM(\Gamma')p_{\Gamma}$ is.

\begin{proof}

  Just as in the proof of Lemma \ref{lem:H2}, if the term $\overset{p}{\underset{\alpha}{\C}}$ appears where $\alpha \leq 0$ then this term is identically zero.

   Set $D = \ell^{\infty}(\Gamma')$.  We rescale all of the weights on $\Gamma$ such that all of the weights on $\Gamma'$ sum to 1.  We have 2 cases: when $\gamma_{v} > \gamma_{w}$ and when $\gamma_{w} \geq \gamma_{v}$.

  \emph{\underline{Case 1, $\gamma_{v} > \gamma_{w}$}}:  We look at the following sequence of inclusions:

   \begin{align*}\cM(\Gamma) \oplus \overset{p_{v}}{\underset{\gamma_{v}}{\C}} \subset \cN_{1} &= \left(\cM(\Gamma) \oplus \overset{p_{v}}{\underset{\gamma_{v}}{\C}}\right) \underset{D}{*} \left(\ell^{\infty}(\Gamma' \setminus \{v, w\}) \oplus \bigoplus_{k=1}^{n} \overset{p_{w, k}}{\underset{\gamma_{w}/n}{\C}} \oplus \left(\bigoplus_{k=1}^{n} \overset{p_{v, k}}{\underset{\gamma_{w}/n}{\C}} \oplus \overset{p_{v}^{I}}{\underset{\gamma_{v} - \gamma_{w}}{\C}}\right)\right)\\
   \cap &\\
  \cN_{2} &= \left(\cM(\Gamma) \oplus \overset{p_{v}}{\underset{\gamma_{v}}{\C}}\right) \underset{D}{*} \left(\ell^{\infty}(\Gamma' \setminus \{v, w\}) \oplus \bigoplus_{k=1}^{n}\overset{p_{w, k}, p_{v, k}}{\underset{\gamma_{w}/n}{M_{2}}} \oplus \overset{p_{v}^{I}}{\underset{\gamma_{v} - \gamma_{w}}{\C}}\right)\\
  \cap &\\
    \cM(\Gamma') &= \left(\cM(\Gamma) \oplus \overset{p_{v}}{\underset{\gamma_{v}}{\C}}\right) \underset{D}{*} \left(\ell^{\infty}(\Gamma' \setminus \{v, w\}) \oplus \underset{2\gamma_{w}}{L(\Z) \otimes M_{2}} \oplus \overset{p_{v}^{I}}{\underset{\gamma_{v} - \gamma_{w}}{\C}}\right).
    \end{align*}
    The projections $p_{w, k}$ are an orthogonal family with trace $\gamma_{w}/n$ in $\cA_{e}$ whose sum is $p_{w}$.  In $\cA_{e}$, $p_{v}$ decomposes as $\sum_{k=1}^{n}p_{v, k} + p_{v}^{I}$ with  $p_{v}^{I}$ supported in the atomic part of $\cA_{e}$, and the $p_{v, k}$ are an orthogonal family of projections with trace $\gamma_{w}/n$ supported in the diffuse part of $\cA_{e}$.
   By the inductive hypothesis,
    $$
    p_{w}\cM(\Gamma)p_{w} = \overset{p^{\Gamma}_{w}}{\underset{\frac{\alpha^{\Gamma}_{w}}{\gamma_{w}}}{L(\F_{t_{w}})}} \oplus \overset{r_{w}^{\Gamma}}{\underset{\frac{\gamma_{w} - \alpha^{\Gamma}_{w}}{\gamma_{w}}}{\C}},
    $$
     with $p_{w}^{\Gamma} = p_{w}p^{\Gamma}$.  We choose $n$ large enough such that $\frac{1}{n} + \frac{\gamma_{w} - \alpha_{w}}{\gamma_{w}} < 1$, i.e., so that $p_{w}\cM(\Gamma)p_{w} * L(\Z / n\Z)$ is an interpolated free group factor.
  From Lemma \ref{lem:DR1},
  $$
  p_{w}\cN_{1}p_{w} = p_{w}\cM(\Gamma)p_{w} * \left(\bigoplus_{k=1}^{n} \overset{p_{w, k}}{\underset{1/n}{\C}}\right) = \left(\overset{p'_{w}}{\underset{\frac{\alpha_{w}}{\gamma_{w}}}{L(\F_{t_{w}})}} \oplus \underset{\frac{\gamma_{w} - \alpha_{w}}{\gamma_{w}}}{\C}\right) * \left(\bigoplus_{k=1}^{n} \overset{p_{w, k}}{\underset{1/n}{\C}}\right),
  $$
   so it is an interpolated free group factor, and applying Lemma $\ref{lem:DR1}$ again, we see that
   $$
   p^{\Gamma}_{w}\cN_{1}p^{\Gamma}_{w} = p^{\Gamma}_{w}\cM(\Gamma)p^{\Gamma}_{w} * p^{\Gamma}_{w}\left[\left(\overset{p^{\Gamma}_{w}}{\underset{\frac{\alpha_{w}}{\gamma_{w}}}{\C}} \oplus \overset{r_{w}^{\Gamma}}{\underset{\frac{\gamma_{w} - \alpha_{w}}{\gamma_{w}}}{\C}}\right) * \left(\bigoplus_{k=1}^{n} \overset{p_{w, k}}{\underset{1/n}{\C}}\right)\right]p^{\Gamma}_{w}
   $$
   with the inclusion $p^{\Gamma}_{w}\cM(\Gamma)p^{\Gamma}_{w} \rightarrow p'_{w}\cN_{1}p'_{w}$ the canonical one.  Therefore $\displaystyle p^{\Gamma}_{w}\cM(\Gamma)p^{\Gamma}_{w} \overset{s.e.}{\hookrightarrow} p^{\Gamma}_{w}\cN_{1}p^{\Gamma}_{w}$, so it follows that $p^{\Gamma}\cM(\Gamma)p^{\Gamma} \overset{s.e.}{\hookrightarrow} p^{\Gamma}\cN_{1}p^{\Gamma}$ as well.  It is clear that $p_{v}^{I}$ will be a minimal central projection in $\cN_{2}$, and since the projections $p_{v, k}$ lie under the minimal projection $p_{v} \in \cM(\Gamma) \oplus \overset{p_{v}}{\underset{\gamma_{v}}{\C}}$, it follows that
   $$
   \cN_{2} = L(\F_{t_{2}}) \oplus \overset{p_{v}^{I}}{\underset{\gamma_{v} - \gamma_{w}}{\C}} \oplus \bigoplus_{u \in B(\Gamma)\setminus\{w\}} \underset{\gamma_{u} - \alpha^{\Gamma}_{u}}{\overset{r^{\Gamma}_{u}}{\C}},
   $$
   where $L(\F_{t_{2}})$ is an amplification of $p^{\Gamma}\cN_{1}p^{\Gamma}$.  Hence  $p^{\Gamma}\cN_{1}p^{\Gamma} = p^{\Gamma}\cN_{2}p^{\Gamma}$.  As a final step, we tensor each copy of $M_{2}$ with $L(\Z)$ to obtain $\cM(\Gamma')$ and apply Lemma \ref{lem:DR1} and Remark \ref{rem:Dyk} $n$ times as in the proof of Lemma \ref{lem:H2} to conclude that
   $$
   \cM(\Gamma') = L(\F_{t_{3}}) \oplus \overset{p_{v}^{I}}{\underset{\gamma_{v} - \gamma_{w}}{\C}} \oplus \bigoplus_{u \in B(\Gamma)\setminus\{w\}} \underset{\gamma_{u} - \alpha^{\Gamma'}_{u}}{\overset{r^{\Gamma}_{u}}{\C}}
   $$
   and $p^{\Gamma}\cN_{2}p^{\Gamma} \overset{s.e.}{\hookrightarrow} p^{\Gamma}\cM(\Gamma')p^{\Gamma}$.  Therefore $p^{\Gamma}\cM(\Gamma)p^{\Gamma} \overset{s.e.}{\hookrightarrow} p^{\Gamma}\cM(\Gamma')p^{\Gamma}$ and $\cM(\Gamma')$ has the desired form.

       \emph{\underline{Case 2, $\gamma_{w} \geq \gamma_{v}$}}:  We look at a sequence of inclusions similar to those in the previous case:
    \begin{align*}
     \cM(\Gamma) \oplus \overset{p_{v}}{\underset{\gamma_{v}}{\C}} \subset \cN_{1} &= \left(\cM(\Gamma) \oplus \overset{p_{v}}{\underset{\gamma_{v}}{\C}}\right) \underset{D}{*} \left(\ell^{\infty}(\Gamma' \setminus \{v, w\}) \oplus \left(\bigoplus_{k=1}^{n} \overset{p_{w, k}}{\underset{\gamma_{v}/n}{\C}} \oplus \overset{p_{w}^{I}}{\underset{\gamma_{w} - \gamma_{v}}{\C}}\right) \oplus \bigoplus_{k=1}^{n} \overset{p_{v, k}}{\underset{\gamma_{v}/n}{\C}}\right)\\
    \cap &\\
    \cN_{2} &= \left(\cM(\Gamma) \oplus \overset{p_{v}}{\underset{\gamma_{v}}{\C}}\right) \underset{D}{*} \left(\ell^{\infty}(\Gamma' \setminus \{v, w\})\oplus \bigoplus_{k=1}^{n} \overset{p_{w, k}, p_{v, k}}{\underset{\gamma_{v}/n}{M_{2}}} \oplus \overset{p_{w}^{I}}{\underset{\gamma_{w} - \gamma_{v}}{\C}}\right)\\
    \cap &\\
    \cM(\Gamma') &= \left(\cM(\Gamma) \oplus \overset{p_{v}}{\underset{\gamma_{v}}{\C}}\right) \underset{D}{*} \left(\ell^{\infty}(\Gamma' \setminus \{v, w\})\oplus \underset{2\gamma_{v}}{L(\Z)\otimes M_{2}} \oplus \overset{p_{w}^{I}}{\underset{\gamma_{w} - \gamma_{v}}{\C}}\right).
    \end{align*}
  The projections $p_{v, k}$ are an orthogonal family with trace $\gamma_{v}/n$ in $\cA_{e}$ whose sum is $p_{v}$.  In $\cA_{e}$, $p_{w}$ decomposes as $\sum_{k=1}^{n}p_{w, k} + p_{w}^{I}$ where $p_{w}^{I}$ is supported in the atomic part of $\cA_{e}$, and the $p_{w, k}$ are an orthogonal family of projections with trace $\gamma_{v}/n$ supported in the diffuse part of $\cA_{e}$.  We choose $n$ large enough so that $\frac{\gamma_{w} - \alpha_{w}}{\gamma_{w}} + \frac{\gamma_{v}}{n\gamma_{w}} < 1$.   Observe by the condition on $n$ that $p_{w}\cN_{1}p_{w} = \overset{p^{\cN_{1}}_{w}}{L(\F_{t'_{1}})} \oplus \underset{\gamma_{w} - \alpha_{w} - \gamma_{v}}{\C}$ where the copy of $\C$ is orthogonal to each $p_{w, k}$.  Therefore as in the proof of Lemma \ref{lem:H2} $p_{\Gamma}\cM(\Gamma)p_{\Gamma} \overset{s.e.}{\hookrightarrow} p_{\Gamma}\cN_{1}p_\Gamma$.  We next look at
   $$
   \cN_{1} \subset \cN_{2} = \left(\cM(\Gamma) \oplus \overset{p_{v}}{\underset{\gamma_{v}}{\C}}\right) \underset{D}{*} \left(\ell^{\infty}(\Gamma' \setminus \{v, w\})\oplus \bigoplus_{k=1}^{n} \overset{p_{w, k}, p_{v, k}}{\underset{\gamma_{v}/n}{M_{2}}} \oplus \overset{p_{v}^{I}}{\underset{\gamma_{w} - \gamma_{v}}{\C}}\right).
   $$
    Since the $p_{v,k}$ lie under the minimal central projection $p_{v} \in \cM(\Gamma) \oplus \overset{p_{v}}{\underset{\gamma_{v}}{\C}}$, the arguments above imply
    $$
    \cN_{2} = L(\F_{t_{2}}) \oplus \underset{\gamma_{w} - \gamma_{v} - \alpha_{w}}{\C} \oplus \bigoplus_{u \in L(\Gamma)\setminus\{w\}} \underset{\gamma_{u} - \alpha_{u}}{\overset{r^{\Gamma}_{u}}{\C}}
    $$
    and $p^{\Gamma}\cN_{1}p^{\Gamma} = p^{\Gamma}\cN_{2}p^{\Gamma}$.  To finish, we tensor each copy of $M_{2}$ with $L(\Z)$ and apply Lemma \ref{lem:DR1} and Remark \ref{rem:Dyk} $n$ times as in the end of the proof of Lemma \ref{lem:H2} to obtain
    $$
    \cM(\Gamma') = L(\F_{t_{3}}) \oplus \underset{\gamma_{w} - \gamma_{v} - \alpha_{w}}{\C} \oplus \bigoplus_{u \in L(\Gamma)\setminus\{w\}} \underset{\gamma_{u} - \alpha_{u}}{\overset{r^{\Gamma}_{u}}{\C}}
    $$
     with the inclusion $p^{\Gamma}\cM(\Gamma)p^{\Gamma} \rightarrow p^{\Gamma}\cM(\Gamma')p^{\Gamma}$ standard. \end{proof}

\begin{proof}[Proof of Theorem \ref{thm:H1}]

Note that if $\Gamma'$ and $\Gamma$ are connected, loopless, finite graphs, then $\Gamma'$ can be constructed form $\Gamma$ by considering the steps in Lemmas $\ref{lem:H2}$ and $\ref{lem:H3}$.  Therefore, we can deduce Theorem $\ref{thm:H1}$ by observing that the composite of standard embeddings is a standard embedding and that standard embeddings are preserved by cut-downs by projections. \end{proof}

%%%%%%%%%%%%%%%%%%%%%%%%%%%%%%%%%%%%%%%%%%%%%%%%%%%%%

\section{The GJS construction in infinite depth} \label{sec:GJSInfinite}

    Recall that the vertices on a principal graph for $A_{0} \subset A_{1}$ represent isomophism classes of irreducible $A_{0}-A_{0}$ and $A_{0}-A_{1}$ subbimodules  of tensor products of $X = _{A_{0}}L^{2}(A_{1})_{A_{1}}$ and its dual, $X^{*} = _{A_{1}}L^{2}(A_{1})_{A_{0}}$.  Assume $\Gamma$ is the principal graph for an infinite-depth subfactor.  If $*$ is the depth-0 vertex of $\Gamma$, then the factor $A_{0}$ as in the introduction is isomorphic to $p_{*}\cM(\Gamma)p_{*}$.    $\cM(\Gamma)$ is now a semifinite algebra where the weighting $\gamma$ on $\ell^{\I}(\Gamma)$ corresponds to the bimodule dimension obtained by identifying each vertex with an irreducible bimodule as above.  Under this identification, $\gamma_* = 1$ and $\delta\cdot\gamma_{v} = \sum_{w \sim v}n_{v, w}\gamma_{w}$ where $\delta = [A_{1}:A_{0}]^{1/2}$.  To circumvent the difficulty of dealing with a semifinite algebra, we realize that $A_{0}$ is an inductive limit of the algebras $p_{*}\cM(\Gamma_{k})p_{*}$ where $\Gamma_{k}$ is $\Gamma$ truncated at depth $k$.  To aid our computation of the isomorphism class of $A_{0}$, we have the following lemma, whose proof is a routine calculation and is identical to that in $\cite{MR2807103}$.

\begin{lem} \label{lem:GJS2}

The free dimension of $\cM(\Gamma_{k})$ is
$$
1 + \frac{1}{\Tr(F_{k})^{2}}\left( -\sum_{v \in \Gamma_{k}} \gamma_{v}^{2} + \sum_{v \in \Gamma_{k}}\sum_{w\sim v} n_{v, w}\gamma_{v}\gamma_{w}\right)
$$  where $w \sim v$ means $w$ is connected to $v$ in $\Gamma_{k}$, $F_{k} = \sum_{u \in \Gamma_{k}} p_{u}$, and $\Tr$ is the trace on the semifinite algebra $\cM(\Gamma)$.

\end{lem}

\begin{thm} \label{thm:H2}

Let $\cP$ be an infinite depth subfactor planar algebra.  Then the factor $A_{0}$ in the construction of \cite{MR2732052} is isomorphic to $L(\F_{\I})$.

\end{thm}

\begin{proof}

For a given $k$, we write
$$
\cM(\Gamma_{k}) = \overset{p_{k}}{\underset{\underset{w \not\in B(\Gamma_{k})}{\sum\gamma_{w}} + \underset{v \in B(\Gamma_{k})}{\sum \alpha^{\Gamma_{k}}_{v}}}{L(\F_{t_{k}})}} \oplus \bigoplus_{v \in B(\Gamma_{k})} \overset{p^{\Gamma_{k}}_{v}}{\underset{\gamma_{v} - \alpha^{\Gamma_{k}}_{v}}{\C}}.
$$
The free dimension of this algebra is
$$
1 + (t_{k}-1)\left(\frac{\sum_{w \not\in B(\Gamma_{k})}\gamma_{w} + \sum_{v \in B(\Gamma_{k})}\gamma_{v}}{\Tr(F_{k})}\right)^{2} - \frac{\sum_{v\in B(\Gamma_{k})}(\gamma_{v} - \alpha^{\Gamma_{k}}_{v})^{2}}{\Tr(F_{k})^{2}},
$$
so by Lemma \ref{lem:GJS2}, we have the equation
$$
(t_{k}-1)\left(\sum_{w \not\in B(\Gamma_{k})}\gamma_{w} + \sum_{v \in B(\Gamma_{k})}\alpha^{\Gamma_{k}}_{v}\right)^{2} =  \sum_{u \in \Gamma_{k}}\sum_{w\sim u} n_{u, w}\gamma_{u}\gamma_{w} -\sum_{u \in \Gamma_{k}}\gamma_{u}^{2} + \sum_{v\in B(\Gamma_{k})}(\gamma_{v} - \alpha^{\Gamma_{k}}_{v})^{2}.
$$
Observe that in $\Gamma_{k}$, the vertices up to depth $k-1$ are connected to all of their neighbors in $\Gamma$, so by the Perron-Frobenius condition and the fact that $\delta > 1$, none of these vertices are in $B(\Gamma_{k})$.  If we let $B'(\Gamma_{k})$ be the vertices $v$ at depth $k$ with $\gamma_{v} \leq \sum_{w\sim v}n_{v, w}\gamma_{w}$, then the right hand side of the above equality becomes

\begin{align*}
(\delta - 1)\sum_{v \in \Gamma_{k-2}}\gamma_{v}^{2} &+ \sum_{v \in \Gamma_{k-1}\setminus\Gamma_{k-2}}\gamma_{v} \left(-\gamma_{v} + \sum_{\substack{ w\not\in B(\Gamma_{k}) \\ w \sim{v}}} n_{v, w}\gamma_{w} + \sum_{\substack{w \in B(\Gamma_{k}) \\ w \sim v}} \alpha^{\Gamma_{k}}_{w} \right) \\  &+ \sum_{v \in B'(\Gamma_{k})} \gamma_{v}\left(-\gamma_{v} + \sum_{w \sim v}n_{v,w}\gamma_{w}\right)
\end{align*}
where we have used $\alpha^{\Gamma_{k}}_{v}  = \sum_{w\sim v}n_{v, w}\gamma_{w}$.  This quantity majorizes $(\delta - 1)\sum_{v \in \Gamma_{k-2}}\gamma_{v}^{2}$. Since the bimodule dimensions of any irreducible sumbimodule of $(X \otimes _{A_{1}} X^{*})^{\otimes_{A_{0}}^{n}}$ and $(X \otimes _{A_{1}} X^{*})^{\otimes_{A_{0}}^{n}} \otimes_{A_{0}} X$ are bounded below by 1, $\gamma_{v} \geq 1$ for all $v \in \Gamma$ so we conclude that
$$
(t_{k}-1)\left(\sum_{w \not\in L(\Gamma_{k})}\gamma_{w} + \sum_{v \in L(\Gamma_{k})}\alpha^{\Gamma_{k}}_{v}\right)^{2} \rightarrow \infty
$$
as $k \rightarrow \infty$.  From the amplification formula, $p_{*}\cM(\Gamma_{k})p_{*} = L(\F_{t'_{k}})$ where
$$
t'_{k} = 1 + (t_{k}-1)\left(\sum_{w \not\in L(\Gamma_{k})}\gamma_{w} + \sum_{v \in L(\Gamma_{k})}\alpha^{\Gamma_{k}}_{v}\right)^{2}.
$$  Hence $p_{\Gamma_{k}}\cM(\Gamma_{k})p_{\Gamma_{k}} \overset{s.e.}{\hookrightarrow} p_{\Gamma_{k}}\cM(\Gamma_{k+1})p_{\Gamma_{k}}$ so by Remark \ref{rem:Dyk}, $p_{*}\cM(\Gamma_{k})p_{*} \overset{s.e.}{\hookrightarrow} p_{*}\cM(\Gamma_{k+1})p_{*}$.  As $p_{*}\cM(\Gamma)p_{*}$ is the inductive limit of the $p_{*}\cM(\Gamma_{k})p_{*}$, it follows that $p_{*}\cM(\Gamma)p_{*} = L(\F_{t})$ where $t = \lim t'_{k} = \infty$. \end{proof}

\begin{cor} The factors $A_{k}$ are isomorphic to $L(\F_{\I})$.

\end{cor}

\begin{proof}  If $k$ is even, then $A_{k}$ is an amplification of $A_{0}$ so it follows for $A_{k}$.  If $k$ is odd, then $A_{k}$ are cut-downs of $\cM(\Gamma')$ with $\Gamma'$ the dual principal graph of $\cP$.  Applying the same analysis as in Theorem \ref{thm:H2} shows that $A_{k} \cong L(F_{\infty})$. \end{proof}

%%%%%%%%%%%%%%%%%%%%%%%%%%%%%%%%%%%%%%%%%%%%%%%%%
\bibliographystyle{amsalpha}
\bibliography{InfiniteDepthGJS}
\end{document}